\documentclass[11pt]{amsart}

\usepackage{amsfonts, amstext, amsmath, amsthm, amscd, amssymb}

\usepackage{graphicx, color}

\usepackage{microtype}

\usepackage[hidelinks,pagebackref,pdftex]{hyperref}

\newtheorem{theorem}{Theorem}[section]
\newtheorem{proposition}[theorem]{Proposition}
\newtheorem{lemma}[theorem]{Lemma}
\newtheorem{corollary}[theorem]{Corollary}

\newtheorem*{namedtheorem}{\theoremname}
\newcommand{\theoremname}{testing}
\newenvironment{named}[1]{\renewcommand{\theoremname}{#1}\begin{namedtheorem}}{\end{namedtheorem}}

\theoremstyle{definition}
\newtheorem{definition}[theorem]{Definition}

\title[Generalized Augmented Braids]{Geometry of Generalized Augmented Braids}
\author{Thiago de Paiva}
\address[]{School of Mathematics, Monash University, VIC 3800, Australia }
\email[]{thiago.depaivasouza@monash.edu}

\begin{document}
 
\begin{abstract}
A generalized augmented link of a knot $K$ is a link obtained by adding trivial components to $K$ that
bound $n$-punctured disks. In this paper we consider that $K$ is given by a positive braid with at least one full twist. 
We characterize when the generalized augmented links of $K$ have hyperbolic geometry.
\end{abstract}  
 
\maketitle

\section{Introduction}

Consider a knot $K$ and $P$ its projection plane. Let $C_1, \dots, C_r$ be a collection of unlinked circles $C_1, \dots, C_r$ that transversely intersect $P$ into two points and are not isotopic to each other in the exterior of $K$ with each $C_i$ encircling $r_i$ strands of $K$. The link $L = K \cup C_1 \cup \dots \cup C_r$ is called a generalized augmented link of the knot $K$.  

Given a knot $K$, the study of the generalized augmented links $L = K \cup C_1 \cup \dots \cup C_r$ of $K$ can be useful to understand $K$ and knots obtained from $K$ by applying full twists along the trivial components $C_1, \dots, C_r$. 
For example, if $K$ is a hyperbolic knot and a generalized augmented link $L$ of $K$ is also hyperbolic, then the hyperbolic volume of $L =  K \cup C_1 \cup \dots \cup C_r$ is an upper bound of the volumes of $K$ and knots obtained from $K$ by applying full twists along $C_1, \dots, C_r$ as the hyperbolic volume reduces under Dehn filling. Also when we know the geometric type of $L$, we may find the geometric types of $K$ and knots obtained by applying full twists along the trivial components $C_1 \cup \dots \cup C_r$, for example see Corollary~\ref{firstcorollary} and Purcell \cite[Theorem 3.2]{MR2660459}.
See Blair, Futer and Tomova\cite{MR3361143}, Futer, Kalfagianni and Purcell \cite{MR2660569}, Lackenby \cite{MR2018964}, Purcell \cite{MR2471374}, Dasbach and Tsvietkova \cite{MR3391876} for more applications of the study of  generalized augmented links.


The exterior of $L$ has the advantage that it can be simplified by removing the ``full twists'' of $K$ by applying full twists  along the trivial components $C_1, \dots, C_r$. Denote by $K'$ the knot that we obtain  after this ``simplification''.
In the paper \cite{MR3709649} Adams studies when $K'$ is an alternating knot. However, if we allow the ``simplified knot'' $K'$ to be not necessarily alternating, then it has the advantage of not adding too many circles to the generalized augmented link $L$.
The same problem can occur if the trivial components can only encircle two strands of $K$, as it has been done in some previous works.
This condition might be sufficient to study some families of knots, see for example \cite{MR2796634}, but it may not be effective to study other families of knots, for example Lorenz links and twisted torus links.
Lorenz links form a family of links that are periodic orbits of the Lorenz system. See \cite{MR4651895}, \cite{newtwis}, \cite{MR4494619}, \cite{MR4576328}, \cite{de2023hyperbolic} for more information about Lorenz links.
Twisted torus links are torus links with some full twists added along some number of adjacent strands. See \cite{hyperbolicity}, \cite{MR4562562}, \cite{MR4411809}, \cite{MR4551666} for more information about twisted torus links.
Some Lorenz links and twisted torus knots are represented by positive braids with at least one full twist and many full twists in more than two strands in other parts of the braids. In fact, every Lorenz links has a positive braid with at least one full twist as a minimal braid, see \cite[Theorem 5.1]{Knottedperiodicorbits} or \cite[Theorem 3.5]{Lorenzknots}. 
So, if we don't allow the trivial components to encircle more than two strands, then we might need to add many circles to simply $K$ satisfactorily. 




As far as the author knows, the study of the geometry of generalized augmented links $L = K \cup C_1 \cup \dots \cup C_r$  started with Adams in 1986 where he first addressed the case that $K$ is an alternating knot and the circles $C_1, \dots, C_r$ encircle only two strands of $K$ \cite{MR0903861}. Later he considered the case that $K$ is still an alternating knot but now the circles $C_1, \dots, C_r$ can encircle more than two strands of $K$ \cite{MR3709649}. Generalized augmented links with  trivial components encircling more than two strands have been studied in some settings, see for example \cite{MR2672234}, \cite{MR3601570}.
But when the knot $K$ is given by a non-alternating braid, then only the generalized augmented links of torus knots have been considered. Lee first considered the case of generalized augmented links of torus knots with only one trivial component \cite{MR3264672}. de Paiva and Purcell considered the case in which generalized augmented links of torus knots can have more than one trivial component but these trivial components have the restriction of being placed encircling the leftmost strands of the standard braids of torus knots \cite{MR4651895}, \cite{de2023hyperbolic}.
 
\begin{figure}
\centering
\includegraphics[scale=4]{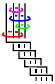} 
\caption{This example illustrates a generalized augmented link of the torus knot $T(4, 5)$.}
\label{Figure}
\end{figure} 
 
Denote by $\sigma_1, \dots, \sigma_{n-1}$ the standard generators of the braid group $B_{n}$.
Consider a positive braid $B$ with $n$ strands. We say that $B$ has $k$ full twists if it has the braid $(\sigma_1\dots\sigma_{n-1})^{kn}$ as a sub-braid with $k>0$.

In this paper we characterize when the generalized augmented links of knots $K$ given by positive braids with at least one  full are hyperbolic. 
More precisely, we study the following case: let $B$ be a positive braid with $n$ strands and at least one full twist representing a knot $K$. Consider a collection of unlinked circles $C_1, \dots, C_r$ that are not isotopic to each other in the exterior of the braid closure of $B$ with each $C_i$ encircling $r_i$ strands of $B$ with $1<r_i\leq n$ such that after applying a positive full twist along each $C_i$ we obtain a positive braid diagram with at least one full twist from $B$. The link $L = K \cup C_1 \cup \dots \cup C_r$ is a generalized augmented link of the knot $K$ as illustrated by an example in Figure~\ref{Figure}. We denote the 3-manifold $\mathbb{S}^3- N(L)$ by $M(K; C_1, \dots, C_r)$. Our main theorem is the following.





\begin{theorem}\label{maintheorem}
Let $B$ be a positive braid with $n$ strands and at least one positive full twist representing a knot $K$. Consider a collection of unlinked circles $C_1, \dots, C_r$ that are not isotopic to each other in the exterior of $K$ with each $C_i$ encircling $r_i$ strands of $B$ with $1<r_i\leq n$ such that after applying a positive full twist along each $C_i$ we obtain a positive braid diagram with at least one full twist from $B$. We denote the link $K \cup C_1 \cup \dots \cup C_r$ by $L$.
\begin{itemize}
\item If $K$ is a hyperbolic knot, then the link $L$ is always hyperbolic.  

\item If $K$ is a torus knot, then the link $L$ is hyperbolic if and only if $r> 1$ or $C_1$ is not the braid axis of $B$.

\item If $K$ is a satellite knot, then the link $L$ is hyperbolic if and only if whenever $K$ is a generalized $q$-cabling of the core of a knotted torus $T$, then $T$ intersects at least one $C_i$. For instance, it happens when there is $a_i > q$ which is not a multiple of $q$.
\end{itemize} 
\end{theorem} 

See Definition~\ref{cabling} for the definition of generalized cabling.

The structure of the paper is as follows: in the first section we characterize when the generalized augmented links of $K$ have no essential spheres, discs, and tori; in the second section we identify when the generalized augmented links of $K$ have no essential annuli; finally in the third section we conclude when the generalized augmented links of $K$ are hyperbolic. 





\subsection{Acknowledgment}
I thank Jessica Purcell for helpful discussions and Tetsuya Ito for helpful comments.

\section{Essential spheres, discs, and tori}

In this section we characterize when $M(K; C_1, \dots, C_r)$ has no essential spheres, discs, and tori. 

\begin{proposition}\label{proposition0}
The manifold $M(K; C_1, \dots, C_r)$ has no essential spheres and essential discs. Therefore, $M(K; C_1, \dots, C_r)$ is irreducible and boundary irreducible.
\end{proposition} 

\begin{proof}Consider that $M(K; C_1, \dots, C_r)$ has an essential 2-sphere $S$. Then, $S$ doesn't bound a ball in $M(K; C_1, \dots, C_r)$. This implies that there is at least one link component of $L=K \cup C_1 \cup \dots \cup C_r$ in each side of $M(K; C_1, \dots, C_r) - S$. Hence there is at least one $C_i$ in a different side of $K$. However this implies that the linking number $lk(K, C_i)$ between $K$ and $C_i$ is zero, a contradiction. So, $M(K; C_1, \dots, C_r)$ has no essential 2-spheres. Therefore, $M(K; C_1, \dots, C_r)$ is irreducible.

Suppose now that $M(K; C_1, \dots, C_r)$ has a boundary reducible disk $D$.
Then, either the boundary of $\partial D$ does not bound a disk on $\partial M(K; C_1, \dots, C_r)$, or
$\partial D$ does bound a disk $E$ on $\partial M(K; C_1, \dots, C_r)$ but $D\cup E$ does not bound a 3-ball in $M(K; C_1, \dots, C_r)$.
It is not possible that $\partial D \subset \partial M(K; C_1, \dots, C_r)$ bounds a disk $E$ on $\partial M(K; C_1, \dots, C_r)$ but $E\cup D$ does not bound a 3-ball, as $M(K; C_1, \dots, C_r)$ is irreducible.
Consider $L_j$ the component of $L$ in which $\partial D$ lies. Then, $\partial D$ does not bound a disk on $\partial N(L_j)$. Thus, $L_j$ would be trivial and $D$ would be its Seifert surface. Since $D$ is not punctured by the other link components of $L$, $L_j$ would have linking numbers equal to zero with the other link components of $L$, a contradiction. Hence $M(K; C_1, \dots, C_r)$ has no essential disks. Therefore, $M(K; C_1, \dots, C_r)$ is also boundary irreducible.
\end{proof}


\begin{definition}\label{cabling}
A generalized $q$-cabling of a link $L$ is a link $L'$ contained in the
interior of a tubular neighbourhood $L\times D^2$ of $L$ such that
\begin{enumerate}
\item each fiber $D^2$ intersects $L'$ transversely in $q$ points;  and

\item all strands of $L'$ are oriented in the same direction as $L$ itself.
\end{enumerate} 
\end{definition}

Denote by $D_{r_i}$ the disc bounded by $C_{r_i}$. 

\begin{lemma}\label{trivial,meridian}
If $M(K; C_1, \dots, C_r)$  has an essential torus $T$ that is not essential when embedded in the exterior of $K$, then $T$ intersects at least one disc $D_{r_i}$ by meridians. 
\end{lemma} 

\begin{proof}
Suppose that $T$ doesn't intersect any $D_{i}$.
Consider that there is one $C_{j}$ in a different side of the manifold $S^{3}-T$ containing $K$. Since the disc $D_{j}$ doesn't intersect $T$, $D_{j}$ is completely contained in the same side that contains $C_{j}$. The disc $D_{j}$ is a Seifert surface for $C_{j}$. But since $D_{j}$ doesn't puncture $K$, this implies that the linking number between $K$ and $C_{j}$ is zero, a contradiction. So, all $C_{j}$ are in the same side as $K$. Hence $T$ is knotted.


Let $D$ be any meridian disc of $T$. Since $T$ is essential in the exterior of $L$, the disc $D$ must intersect $K$ or at least one circle $C_{j}$. Consider that $D$ intersects $C_{j}$ instead of $K$. The disc bounded by $C_{j}$ is completely inside $T$. When we apply a (1)-Dehn surgery along $C_{j}$, the knot $K$ is transformed into a new knot that punctures $D$ at least twice. Therefore, after applying (1)-Dehn filling along  each $C_i$, $B$ is transformed into a new positive braid $B'$ with at least one full twist such that $T$ is essential in the exterior of its closure, called $K'$. By Ito \cite[Theorem 5.5, Lemma 3.5]{ito2024satellite}, $T$ doesn't intersect the braid axis of $B'$. 
So, $T$ lives in the solid torus given by the complement of the braid axis of $B'$. This means that $K'$ is a generalized $q$-cabling of the core of $T$ with $q>1$.
Since $T$ doesn't intersect each $C_i$, if we perform $(-1)$-Dehn filling along each $C_i$, then $K$ would remain a generalized $q$-cabling of the core of $T$, but this contradicts the fact that $T$ is inessential in the exterior of $K$. Hence $T$ intersects at least one $D_{i}$.


Consider $\gamma$ a loop in $D_{i}\cap T$. This loop $\gamma$ bounds a disc $D$ in $D_{i}$, so $\gamma$ is trivial. 
Furthermore, we can assume that $\gamma$ is essential in both $D_{i}$ and $T$. Because if $\gamma$ is essential in $T$ but not in 
$D_{a_i}$, then $D$ is not punctured by $K$ and it would be a compressible disc for $T$, which is not possible. 
Suppose now that $\gamma$ is essential in $D_{i}$ but not in  $T$. Then, we can find a disc $D'$ in $T$ with boundary in $\gamma$. Furthermore, since $\gamma$ is essential in $D_{i}$, $D$ intersects $K$ at least one time. 
The two discs $D$ and $D'$ are both Seifert surfaces for $\gamma$. Then, because of $D$ we conclude that the linking number between $K$ and $\gamma$ is non-zero and because of $D'$ we conclude that the linking number between $K$ and $\gamma$ is zero, a contradiction. Finally consider that $\gamma$ is not essential neither in $D_{a_i}$ nor in $T$. Then, $\gamma$ bounds a disc $D'$ in $T$. 
The two discs $D$ and $D'$ glue together long their boundaries to produce a sphere. This sphere bounds a ball that can be isotopic to miss the disc $D_{i}$ so that the loop $\gamma$ no longer exists in $D_{i}\cap T$. But $T$ must intersect $D_{i}$. So, we can consider that all loops in $D_{i}\cap T$ are essential in both $D_{i}$ and $T$. Additionally, they bound open discs that don't intersect $T$. 
Hence all loops $\gamma$ in $D_{a_i}\cap T$ are either meridians or longitudes of $T$. 
However, we can consider that they are all meridians of $T$.
Because if a loop $\gamma$ in $D_{a_i}\cap T$ is not a meridian of $T$, then $T$ is trivial and $\gamma$ is the longitude of $T$. 
However, $T$ bounds a second solid torus $V$ so that $\gamma$ is the meridian of $\partial V$. So, we can assume that $\gamma$ is a meridian of $T$. Therefore, we can consider that all loops in $D_{a_i}\cap T$ are meridians of $T$.
\end{proof}

\begin{lemma}\label{twice,meridian}
If $M(K; C_1, \dots, C_r)$  has an essential torus $T$ that is not essential when embedded in the exterior of $K$, then $T$ intersects at least twice one disc $D_{r_i}$ by meridians. 
\end{lemma} 

\begin{proof} 
By Lemma~\ref{trivial,meridian}, $T$ intersects at least one disc $D_{r_i}$ by meridians. By contradiction, suppose that $T$ intersects each disc $D_{r_j}$ at most once.

Suppose first that $T$ intersects only one $D_{r_i}$. The link components $K$ and $C_{r_i}$ are in different sides of $M(K; C_1, \dots, C_r)-T$. Suppose there is a different trivial link component $C_{r_j}$ in the same side as $C_{r_i}$. Since $T$ doesn't intersect $D_{r_j}$, this implies that $lk(C_{r_j}, K) = 0$, a contradiction. Hence all trivial link components $C_{r_j}$ different from $C_{r_i}$ are in the same side as $K$. This implies that one side of $M(K; C_1, \dots, C_r)-T$ contains only $D_{r_i}$. If $r = 1$, then $T$ must be knotted otherwise $T$ would be boundary parallel to $N(C_{r_i}),$ which is not possible. This would imply that $T$ is an essential torus for $K$, a contradiction. So $r>1$. But then $T$ becomes an essential torus for the link $K \cup C_1 \cup \dots \cup C_{r_i-1} \cup C_{r_i+1} \cup \dots \cup C_{r}$ that doesn't intersect any $C_1, \dots, C_{r_i-1}, C_{r_i+1}, \dots, C_{r}$ contradicting Lemma~\ref{trivial,meridian}.
 
Consider now that that $T$ intersects at least two $D_{r_i}$ and $D_{r_j}$. Then, $C_{r_i}$ and $C_{r_j}$ are isotopic to a meridian of $T$. But this implies that $C_{r_i}$ and $C_{r_j}$ are isotopic to each other contradicting the assumption that any two trivial link components $C_{r_i}$ and $C_{r_j}$ are not isotopic to each other in the exterior of $K$.

Therefore, we conclude that $T$ intersects at least twice one disc $D_{r_i}$ by meridians. 
\end{proof}

\begin{lemma}\label{noessentialtorus}
There is no essential torus $T$ in $M(K; C_1, \dots, C_r)$ that is not essential when embedded in the exterior of $K$.
\end{lemma} 

\begin{proof}
Suppose that $M(K; C_1, \dots, C_r)$ has an essential torus $T$ that is not essential when embedded in the exterior of $K$. Then, $T$ intersects one disc $D_{r_i}$ at least twice by meridians by Lemma~\ref{twice,meridian}.

The intersection between $T$ and $D_{r_i}$ is formed by circles $T_{C^1}, \dots, T_{C^u}$ with $u\geq 2$. Denote by $T_{D^1}, \dots, T_{D^u}$ the discs that $T_{C^1}, \dots, T_{C^u}$ bound, respectively. We have that $D_{r_i} \cap T(p, q) = (T_{D^1}\cup \dots \cup T_{D^u})\cap T(p, q)$.
Furthermore, each $T_{D^j}$ intersects $K$ in the same amount of points for all $j$ because the linking number between $T_{C^j}$ and $K$ must be persevered. So, the number of points in $T_{D^j}\cap K$ is constant for all $j$. 

Consider that each $T_{D^j}$ intersects $K$ once. We apply a Dehn filling along $C_j$ with high positive slope, if necessary, to transform $T$ into a knotted torus $T'$. The braid $B$ is transformed into a new positive braid $B'$ with at least one full twist. By Ito \cite[Theorem 5.5, Lemma 3.5]{ito2024satellite}, $T'$ doesn't intersect the braid axis of $B'$. This means that the closure of $B'$ is a generalized $q$-cabling of the core of $T'$ with $q>1$. But this contradicts the fact that $T_{D^j}$ intersects $K$ once, since $T_{D^j} \cap K$ should be formed by $q$ points and $q>1$.


Suppose now that each $T_{D^j}$ intersects $K$ in $q$ points with $1<q<r$. We apply a Dehn filling along $C_i$ with high positive slope, if necessary, to transform $T$ into a knotted torus $T'$. The braid $B$ is transformed into a new positive braid $B'$ with at least one full twist. By \cite{ito2024satellite}, $T'$ doesn't intersect the braid axis of $B'$.  This means that the closure of $B'$ is a generalized $q$-cabling of the core of $T'$ with $q>1$. As a consequence, this implies that $K$ is a generalized $q$-cabling of the core of $T$. So $T$ is unknotted, as it is inessential in the exterior of $K$. But after deforming $T$ to the trivial torus with one strand, the knot $K$ is given by a braid with fewer strands than the number of strands in $B$, but this is a contradiction since $B$ is a minimal braid for $K$ because a positive braid with at least one positive full twist is a minimal braid by Franks and Williams \cite[Corollary 2.4]{Franks}.
\end{proof}

\begin{lemma}\label{torusandhyperbolic}
If $K$ is a torus or hyperbolic knot, then $M(K; C_1, \dots, C_r)$ is atoroidal.
\end{lemma} 

\begin{proof} 
By Lemma~\ref{noessentialtorus}, if $M(K; C_1, \dots, C_r)$ has an essential torus $T$, then $T$ is essential in the exterior of $K$. However, this is not possible because torus and hyperbolic knots don't have any essential torus in their exteriors.  
\end{proof} 

\begin{lemma}\label{satellite}
If $K$ is a satellite knot, then $M(K; C_1, \dots, C_r)$ is atoroidal if and only if whenever $K$ is a generalized $q$-cabling of the core of a knotted torus $T$, then $T$ intersects at least one $C_i$.
\end{lemma} 

\begin{proof} 
We prove this lemma by proving its contrapositive which says that $M(K; C_1, \dots, C_r)$ is toroidal if and only if $T$ doesn't intersect any $C_i$ whenever $K$ is a generalized $q$-cabling of the core of a knotted torus $T$.

Consider first that $K$ is satellite knot that is a generalized $q$-cabling of the core of a knotted torus $T$ that doesn't intersect any $C_i$, then $T$ is also an essential torus for $M(K; C_1, \dots, C_r)$. Hence $M(K; C_1, \dots, C_r)$ is toroidal.

Consider now that $M(K; C_1, \dots, C_r)$ has an essential torus $T$.
The torus $T$ is  essential when embedded in the exterior of $K$ by Lemma~\ref{noessentialtorus}.
Since $B$ is a positive braid with at least one full twist, $T$ doesn't intersect the braid axis of $B$ by \cite{ito2024satellite}. Hence $K$ is a generalized $q$-cabling of the core of $T$ with $q>1$. As a consequence, $T$ doesn't intersect any $C_i$, as we want.
\end{proof}

\begin{proposition}\label{atoriodal}
Let $B$ be a positive braid with $n$ strands and at least one positive full twist representing a knot $K$. Consider a collection of unlinked circles $C_1, \dots, C_m$ that are not isotopic to each other in the exterior of the braid closure of $B$ with each $C_i$ encircling $r_i$ strands of $B$ with $1<r_i\leq n$ such that after applying a positive full twist along each $C_i$ we obtain a positive braid diagram with at least one full twist from $B$. We denote the link $K \cup C_1 \cup \dots \cup C_m$ by $L$.
\begin{itemize}
\item If $K$ is a torus or hyperbolic knot, then the link $L$ is always atoroidal.  

\item If $K$ is a satellite knot, then the link $L$ is atoroidal if and only if whenever $K$ is a generalized $q$-cabling of the core of a knotted torus $T$, then $T$ intersects at least one $C_i$.
\end{itemize} 
\end{proposition} 

\begin{proof} 
It follows from Lemmas~\ref{torusandhyperbolic} and ~\ref{satellite}.
\end{proof}

\section{Essential annuli}

In this section we characterize when $M(K; C_1, \dots, C_r)$ has no essential annuli. 


\begin{lemma}\label{proposition1}
$M(K; C_1, \dots, C_r)$ contains an essential annulus with one boundary component in $\partial N(K)$ and the other in $\partial N(C_i)$ if and only if $K$ is a torus knot, $r = 1$ and $C_1$ is the braid axis of $B$.
\end{lemma}

\begin{proof}
Suppose $K$ is a torus knot, $r = 1$ and $C_1$ is the braid axis of $B$. The braid $B$ is a positive braid with $n$ strands and at least one full twist. This implies that $B$ is a minimal braid \cite[Corollary 2.4]{Franks} and $K$ is the torus knot $T(n, q)$ with $n$ being the amount of strands of $B$ and $q>n$. There is an isotopy in the complement of the braid axis $C_1$ that takes $K$ to $T(n, q)$ by Los \cite{Los}. Then, we place $T(n, q)$ into $\partial N(C_1)$ by applying the isopoty that changes the torus knot $T(n, q)$ to the torus knot $T(q, n)$. We push the torus knot $T(q, n)$ a little out of $\partial N(C_1)$ and consider the annulus with one boundary forming the torus knot $T(q, n)$ in $\partial N(C_1)$ and the other boundary wrapping one time along the longitude of $\partial N(T(q, n))$. This annulus is an essential annulus in $M(K; C_1)$.

Consider now that $M(K; C_1, \dots, C_r)$ contains an essential annulus $A$ with one boundary component $\partial_1 A$ in $\partial N(C_i)$ and the other $\partial_2 A$ in $\partial N(K)$. Since $C_i$ is the unknot, the knot $\partial_1 A$ is a $(a, b)$-torus knot in $\partial N(C_i)$. Suppose first that $\partial_1 A$ is a trivial torus knot. 
Since $K$ is a non-trivial knot and $\partial_1 A$ and $\partial_2 A$ are isotopic through $A$, then $\partial_2 A$ is the meridian of $\partial N(K)$. This implies that the linking numbers $lk(\partial_2 A, K)$ and $lk(\partial_1 A, K)$ are equal to $1$. Since $lk(C_i, K) >1$, then $\partial_1 A$ can't wrap any time along the longitude of $\partial N(C_i)$. Hence $\partial_1 A$ is the meridian of $\partial N(C_i)$. However, this implies that $k(\partial_1 A, K)$ is equal to zero, a contradiction. So, $\partial_1 A$ is a non-trivial torus knot. 

Since $\partial_1 A$ can be embedded in $\partial N(K)$ and torus knots don't have an essential torus by Tsau \cite{Incompressible}, the knot $\partial_2 A$ wraps only one time along the longitude of $\partial N(K)$. Consider there is a circle $C_j$ in $C_1, \dots, C_r$ different from $C_i$. From $lk(C_j, K) \neq 0$ we conclude that $lk(C_j, \partial_1 A) \neq 0$ and from the fact that $C_j$ and $C_i$ are unlinked we conclude that $lk(C_j, \partial_1 A) = 0$, a contradiction. So, $r = 1$.

Because $\partial_2 A$ and $K$ are isotopic, this implies that $K$ can be embedded in $\partial N(C_1)$. 
So, $K$ is a torus knot of the form $T(n, q)$ with $n$ being the number of strands of $B$ and $q>n$. The knot $K$ is the torus knot $T(n, q)$ or $T(q, n)$ in $\partial N(C_1)$. We have that $K$ can't be the torus knot $T(n, q)$ in $\partial N(C_1)$ as it would imply that $lk(C_1, K)=q>n$, a contradiction. So, $\partial_1 A$ is the torus knot $T(q, n)$ in $\partial N(C_1)$. Hence $lk(C_1, K) = n$. Therefore, $C_1$ is the braid axis of $B$, as we want.
\end{proof}

\begin{lemma}\label{annulitwocomponent3}
$M(K; C_1, \dots, C_r)$ contains no essential annuli with boundary components in $C_i$ and $C_j$ with $i\neq j$.
\end{lemma}

\begin{proof}
Consider $A$ an annulus in $M(K; C_1, \dots, C_r)$ such that $\partial_1 A\subset \partial N(C_i)$ and $\partial_2 A\subset  \partial N(C_j)$.
We have that $lk(\partial_2 A, C_i) = 0$ because $C_i$ and $C_j$ are unlinked. Since $\partial_1 A$ and $\partial_2 A$ are isotopic, we also have that $lk(\partial_1 A, C_i) = 0$.
Therefore, $\partial_1 A$ is the longitude of $\partial N(C_i)$. For the same reason, $\partial_2 A$ is the longitude of $\partial N(C_j)$. So, $\partial_1 A$ is isotopic to $C_i$ and $\partial_2 A$ is isotopic to $C_j$.
Thus, $C_i$ and $C_j$ are isotopic through $A$. However, this contradicts the assumption that $C_i$ and $C_j$ are not isotopic in the exterior of $K$.
\end{proof}

\begin{lemma}\label{twocomponent3}
Consider that $K$ is not a torus knot or $K$ is a torus knot and $m > 1$ or $C_1$ is not the braid axis of $B$.
Then, $M(K; C_1, \dots, C_r)$ contains no essential annuli with boundary components in two different boundary components of $M(K; C_1, \dots, C_r)$.
\end{lemma}

\begin{proof}
It follows from Lemmas~\ref{proposition1} and \ref{annulitwocomponent3}.
\end{proof}

\begin{lemma}\label{annulusinK}
$M(K; C_1, \dots, C_r)$ has an essential annulus with boundaries in $K$ that is also essential in the exterior of $K$ if and only if $K$ is a torus knot, $m=1$ and $C_1$ is the braid axis of $B$ or $K$ is a cable knot on a knot $P$ such that the boundary of the tubular neighbourhood of $P$ doesn't intersect any $C_i$.
\end{lemma}

\begin{proof}
Consider that $K$ is a torus knot, $r=1$ and $C_1$ is the braid axis of $B$. Then, we embed $K$ into a trivial torus $T$ boundary parallel to $\partial N(C_1)$. Then, $T-K$ is an essential annulus in the exterior of $K$ that also lives in $M(K; C_1)$. 

Consider now that $K$ is a cable knot on a knot $P$ such that the boundary of the tubular neighbourhood of $P$ doesn't intersect any $C_i$. We embed $K$ into $\partial N(P)$. Then, $\partial N(P)-K$ is an essential annulus in the exterior of $K$ that also lives in $M(K; C_1, \dots, C_r)$. Hence $M(K; C_1, \dots, C_r)$ has an essential annulus in this case as well. 

Consider now that $M(K; C_1, \dots, C_r)$ has an essential annulus $A$ with boundaries in $K$ that is also essential in the exterior of $K$. So, $K$ is either a composite, cable, or torus knot. Since $K$ is the closure of a positive braid with a full twist, $K$ can't be a composite knot by Michael \cite{MR1457195}. 

Assume now that $K$ is a cable knot on a knot $P$. In this case, there is an annulus $B$ in $\partial N(K)$ that is glued with $A$ along their boundaries to form the essential torus $\partial N(P)$ in $M(K; C_1, \dots, C_r)$.  Hence $\partial N(P)$ doesn't intersect any $C_i$.

Finally consider that $K$ is a torus knot. There is an annulus $B$ in $\partial N(K)$ that is glued with $A$ along their boundaries to form an unknotted torus $T$ in which $K$ is embedded. All $C_i$ are in one solid torus bounded by $T$. Denote by $\omega_i$ the winding number of $C_i$ in the solid torus that it lies. The knot $K$ is a torus knot with parameters $(n, q)$ with $n$ being the number of strands of $B$ and $q>n$. The linking number $lk(C_i, K)$ between $C_i$ and $K$ is equal to $\omega_in$ or $\omega_iq$. By definition $lk(C_i, K)$ is at most $n$. So, $lk(C_i, K) = n$. Hence there is only one circle in $C_1, \dots, C_r$ and this circle $C_1$ is the braid axis of $B$.
\end{proof}

\begin{lemma}\label{proposition2}
Consider that if $K$ is a torus knot, then $r>1$ or $r=1$ and $C_1$ is not the braid axis of $B$, or if $K$ is a cable knot on a knot $P$, then the boundary of the tubular neighbourhood of $P$ intersects at least one $C_i$.
If $M(K; C_1, \dots, C_r)$ has an essential annulus $A$, then $\partial A$ are contained in just one boundary component of 
$\partial M(K; C_1, \dots, C_r)$. Furthermore, $A$ is inessential in the exterior of this boundary component.
\end{lemma} 

\begin{proof} 
Suppose that $M(K; C_1, \dots, C_r)$ has an essential annulus $A$.
By Lemma~\ref{twocomponent3}, $\partial A$ are contained in just one of $\partial_{K}M$ or $\partial_{C_i} M$ with $i=1, \dots, r$. Call $B$ this boundary component in which $\partial A$ lie. 
If $B = \partial_{C_i} M$, then $A$ is not essential in the exterior of $B$ because the unknot torus has no essential annuli \cite[page 15]{hatcher2007notes}. If $B$ is equal to $\partial_{K}M$, then $A$ is inessential in the exterior of $B$ due to Lemma~\ref{annulusinK}.
\end{proof}
 
\begin{proposition}\label{proposition3}
Consider that if $K$ is a torus knot, then $r>1$ or $r=1$ and $C_1$ is not the braid axis of $B$, or if $K$ is a cable knot on a knot $P$, then the boundary of the tubular neighbourhood of $P$ intersects at least one $C_i$.
Then, $M(K; C_1, \dots, C_r)$ has no essential annulus $A$ with boundaries in one $\partial_{C_i} M(K; C_1, \dots, C_r)$.
\end{proposition} 

\begin{proof} 
Suppose $M(K; C_1, \dots, C_r)$ contains an essential annulus $A$ with boundaries in one $\partial_{C_i} M(K; C_1, \dots, C_r)$. We denote by $M_{C_i}$ the exterior of $C_i$ in $S^3$. By Lemma~\ref{proposition2}, $A$ is inessential in $M_{C_i}$.
Hence $A$ is compressible, boundary compressible, or boundary parallel in $M_{C_i}$.

If $A$ is boundary parallel to an annulus $A'$ in $\partial N(C_i)$, then $A\cup A'$ bounds a solid torus $V$ in $M_{C_i}$. Because  $A$ is not boundary parallel in $M(K; C_1, \dots, C_r)$, a different boundary component of $M(K; C_1, \dots, C_r)$ must be inside $V$. Suppose first that $K$ is inside $V$. The wrapping number of $K$ inside $V$ must be different from zero as the linking number between $K$ and $C_i$ is different from zero. Furthermore, $K$ can't be the core of $V$ otherwise $K$ could be embedded in $\partial N(C_i)$ implying that  $M(K; C_1, \dots, C_r)$ has an essential annulus with boundaries in $K$ and $C_i$, which is not possible due to Lemma~\ref{twocomponent3}.

The boundaries of $A$, $\partial_1 A$ and $\partial_2 A$, are isotopic to each other and each forms a $(a, b)$-torus knots in $\partial N(C_i)$. Since the linking number between $K$ and $C_i$ is different from zero, each $\partial_i A$ is the torus knot $T(a, b)$ with $b\neq 0$. This implies that if there is at least one more component $C_j$, with $i\neq j$, outside $V$, then $\partial V$ is not boundary parallel in $M(K; C_1, \dots, C_r)$ and so essential, but this is not possible due to Proposition~\ref{atoriodal}. 

So all components different from $C_i$ are inside $V$. Hence each $\partial_i A$ can't be of the form $T(a, \pm 1)$ otherwise the disc which bounds the longitude of $\partial N(C_i)$ is a boundary compression disk for $A$, which implies that $A$ is not essential in $M(K; C_1, \dots, C_r)$, a contradiction. So  each $\partial_i A$ is of the form $T(a, b)$ with $\vert b \vert >1$. But this implies that $\partial V$ is an essential torus for $M(K; C_1, \dots, C_r)$, a contradiction. We conclude that $K$ can't be inside $V$. 

Suppose now that a component $C_j$ different from $C_i$ must be inside $V$. In this case we expand $\partial N(C_i)$ to include $C_j$ so that we obtain an essential trivial torus $T$ containing $C_i$ and $C_j$ in the same inside and $K$ in the other side, but this is a contradiction as $M(K; C_1, \dots, C_r)$ don't have an essential torus.



Consider now that $A$ is compressible in $M_{C_i}$. Then, there is a compression disk $D$ for $A$ in $M_{C_i}$. The surgery of $A$ along $D$ yields two disks, $D_1$ and $D_2$, such that $\partial A = \partial D_1 \cup \partial D_2$.
If one $\partial D_i$ bounds a disk on $\partial N(C_i)$, then by considering a disk with boundary in $A$ close to $\partial D_i$, we can see that $A$ is also compressible in $M(K; C_1, \dots, C_r)$, a contradiction.
Now suppose that none $\partial D_i$ bounds a disk on $\partial N(C_i)$. Then, both $\partial D_1$ and $\partial D_2$ are either isotopic to the longitude or meridian of $\partial N(C_i)$ implying that $A$ is also boundary parallel in $M_{C_i}$. However, we have already ruled out this possibility.
 
Therefore, since an incompressible surface inside an irreducible 3-manifold,
with boundaries contained in the torus boundary components of the 3-manifold, is essential
unless the surface is a boundary parallel annulus \cite[page 15]{hatcher2007notes},  $A$ is essential in $M_{C_i}$. However, this contradicts Lemma~\ref{proposition2}.
\end{proof}  
 
\begin{proposition}\label{proposition4}
Consider that if $K$ is a torus knot, then $r>1$ or $r=1$ and $C_1$ is not the braid axis of $B$, or if $K$ is a cable knot on a knot $P$, then the boundary of the tubular neighbourhood of $P$ intersects at least one $C_i$.
Then, $M(K; C_1, \dots, C_r)$ has no essential annulus $A$ with boundaries in $\partial_{K} M(K; C_1, \dots, C_r)$.
\end{proposition} 
  
\begin{proof} 
Suppose $M(K; C_1, \dots, C_r)$ contains an essential annulus $A$ with boundaries in $\partial_{K} M(K; C_1, \dots, C_r)$. We denote by $M_{K}$ the exterior of $K$ in $S^3$. By Lemma~\ref{proposition2}, $A$ is inessential in $M_{K}$.
So $A$ is compressible, boundary compressible, or boundary parallel in $M_{K}$.
  
If $A$ is boundary parallel to an annulus $A'$ in $\partial N(K)$, then $A\cup A'$ bounds a solid torus $V$ in $M_{K}$. Since  $A$ is not boundary parallel in $M(K; C_1, \dots, C_r)$,  one component $C_i$ must be inside $V$. If there is a meridian disk $D$ of $V$ which does not intersect $C_i$, then we conclude that the linking number between $C_i$ and $K$ would be zero, a contradiction. Hence, the wrapping number of $C_i$ in $V$ is greater than zero. 
Because the unknot has no essential torus \cite[page 15]{hatcher2007notes}, $\partial V $ is trivial. This implies that the boundaries of $A$ are isotopic to the meridian of $\partial N(K)$ and $A$ is a meridional annulus of $\partial N(K)$. The wrapping number of $C_i$ inside $V$ must be greater than one because the linking number between $K$ and $C_i$ is greater than one. This implies that $\partial V$ is not boundary parallel to $\partial N(C_i)$. Thus, $\partial V$ is essential in $M(K; C_1, \dots, C_r)$. But this contradicts Proposition~\ref{atoriodal}.

Suppose now that $A$ is compressible in $M_{K}$. Then, there is a compression disk $D$ for $A$ in $M_{K}$. The surgery of $A$ along $D$ yields two disks, $D_1$ and $D_2$, such that $\partial A = \partial D_1 \cup \partial D_2$.
If one $\partial D_i$ bounds a disk on $\partial N(K)$, then by considering a disk with boundary in $A$ close to $\partial D_i$, we can see that $A$ is also compressible in $M(K; C_1, \dots, C_r)$, a contradiction.
Then, none $\partial D_i$ bounds a disk on $\partial N(C_i)$. So each $\partial D_i$ is essential in $\partial N(K)$. But this implies that $K$ is a trivial knot, a contradiction.

Therefore, as $A$ is an incompressible and not boundary parallel surface in $M_{K}$, $A$ is essential in $M_{K}$, which is not possible by Lemma~\ref{proposition2}.
\end{proof} 

Therefore, we conclude the following proposition from this section. 
 
\begin{proposition}\label{annular}
Let $B$ be a positive braid with $n$ strands and at least one positive full twist representing a knot $K$. Consider a collection of unlinked circles $C_1, \dots, C_r$ that are not isotopic to each other in the exterior of the braid closure of $B$ with each $C_i$ encircling $r_i$ strands of $B$ with $1<r_i\leq n$ such that after applying a positive full twist along each $C_i$ we obtain a positive braid diagram with at least one full twist from $B$. We denote the link $K \cup C_1 \cup \dots \cup C_r$ by $L$.
\begin{itemize}
\item If $K$ is a hyperbolic knot, then $L$ is always annular.

\item If $K$ is a torus knot, then $L$ is annular if and only if $r>1$ or $r=1$ and $C_1$ is not the braid axis of $B$.

\item If $K$ is a satellite knot, then $L$ is annular if and only if $K$ is not a cable knot or if $K$ is a cable knot on a knot $P$, then the boundary of the tubular neighbourhood of $P$ intersects at least one $C_i$.
\end{itemize} 
\end{proposition} 

\section{Final classification and corollary}
In this section we obtain Theorem~\ref{maintheorem}. For this, we use the fact that a link complement admits a complete hyperbolic structure if and only if it is irreducible, boundary irreducible, annular, and atoriodal, which is proved by Thurston \cite{Thurston}.

\begin{named}{Theorem~\ref{maintheorem}}
Let $B$ be a positive braid with $n$ strands and at least one positive full twist representing a knot $K$. Consider a collection of unlinked circles $C_1, \dots, C_r$ that are not isotopic to each other in the exterior of the braid closure of $B$ with each $C_i$ encircling $r_i$ strands of $B$ with $1<r_i\leq n$ such that after applying a positive full twist along each $C_i$ we obtain a positive braid diagram with at least one full twist from $B$. We denote the link $K \cup C_1 \cup \dots \cup C_r$ by $L$.
\begin{itemize}
\item If $K$ is a hyperbolic knot, then the link $L$ is always hyperbolic.  

\item If $K$ is a torus knot, then the link $L$ is hyperbolic if and only if $r> 1$ or $C_1$ is not the braid axis of $B$. 

\item If $K$ is a satellite knot, then the link $L$ is hyperbolic if and only if whenever $K$ is a generalized $q$-cabling of the core of a knotted torus $T$, then $T$ intersects at least one $C_i$. For instance, it happens when there is $a_i > q$ which is not a multiple of $q$.
\end{itemize} 
\end{named}

\begin{proof}  
From Proposition~\ref{proposition0}, the link $L$ is always irreducible and  boundary irreducible.

If $K$ is a hyperbolic knot, then the link $L$ is atoriodal by Proposition~\ref{atoriodal} and annular by Proposition~\ref{annular}. Hence $K$ is hyperbolic.

Consider now that $K$ is a torus knot, then the link $L$ is atoriodal by Proposition~\ref{atoriodal}. We conclude from Proposition~\ref{annular} that the link $L$ is hyperbolic if and only if $r\neq 1$ or $C_r$ is not the braid axis of $B$.

Suppose finally that $K$ is a satellite knot. From Proposition~\ref{atoriodal}, the link $L$ is atoroidal if and only if whenever $K$ is a generalized $q$-cabling of the core of a knotted torus $T$, then $T$ intersects at least one $C_i$. From Proposition~\ref{annular}, $L$ is annular if and only if $K$ is not a cable knot or if $K$ is a cable knot on a knot $P$, then the boundary of the tubular neighbourhood of $P$ intersects at least one $C_i$. A cable knot is a generalized cabling but a generalized cabling can not be a cable knot. Therefore, we conclude that the link $L$ is hyperbolic if and only if $T$ intersects at least one $C_i$ whenever $K$ is a generalized $q$-cabling of the core of $T$.
\end{proof} 

\begin{corollary}\label{firstcorollary}
Let $B$ be a positive braid with $n$ strands and at least one positive full twist representing a torus knot $K$. Consider a circle $C_1$ encircling $r_i$ strands of $B$ with $1<r_i < n$ such that after applying a positive full twist along $C_1$ we obtain a positive braid diagram with at least one full twist from $B$. The closure of the braid obtaining from $B$ by adding at least four full twists in either direction along $C_1$ is hyperbolic.  
\end{corollary}

\begin{proof}
Denote by $K'$ the closure of the braid obtaining from $B$ by adding $s$ full twists with $\vert s \vert\geq 4$ along $C_1$. 

Consider that the knot $K$ has an essential torus. The closure of $B$ has an essential annulus. The manifold $M(K; C_1)$ is hyperbolic from Theorem~\ref{maintheorem}. Then, by Gordon and Wu \cite[Theorem 1.1]{GordonToroidal}, $\vert s\vert = 4$ or $5$ and $M(K; C_1)$ is homeomorphic to $M_1$, $M_2$, or $M_3$, with these manifolds being the exteriors of the links $L_1, L_2, L_3$ in \cite[Figure 7.1 (a), (b), and (c)]{GordonToroidal}, respectively.
However, $M(K; C_1)$ is not homeomorphic to any $M_1$, $M_2$, or $M_3$, a contradiction.
Consider now that $K$ has an essential annulus. Then, $\vert s\vert < 3$ by Gordon and Wu \cite[Theorem 1.1]{GordonAnnular}, a contradiction. Lastly, $K$ is not the trivial knot (not boundary irreducible) by Gordon and Wu \cite[Theorem 1]{GordonWu}. Therefore, $K'$ is hyperbolic.
\end{proof}


\bibliographystyle{amsplain}  

\bibliography{References}

\end{document}